\documentclass[12pt,a4paper]{amsart}
\usepackage[a4paper,centering]{geometry}
\usepackage{mathptmx}
\DeclareMathAlphabet{\mathcal}{OMS}{cmsy}{m}{n} % but do not change mathcal symbols

\usepackage{amssymb,authoraftertitle,enumerate,fancyhdr,perpage}

\theoremstyle{plain} % the default
\newtheorem{theorem}{Theorem}
\newtheorem{lemma}{Lemma}

\newcommand{\modulus}[1]{\ensuremath{|\, #1 \,|}}

\MakePerPage[2]{footnote}

\allowdisplaybreaks

\DeclareMathOperator{\Ker}{Ker}

\begin{document}

\title{The Hua identity implies alternativity}

\author{Pasha Zusmanovich}
\address{University of Ostrava, Czech Republic}
\email{pasha.zusmanovich@osu.cz}

\date{First written September 24, 2026}
%\thanks{J. XXX, to appear} 

\begin{abstract}
The well-known Hua identity $a - (a^{-1} + (b^{-1} - a)^{-1})^{-1} = aba$ holds
in any associative division ring. We prove, in a sense, a converse: if a
(nonassociative) ring with unit and multiplicative inverses satisfies the Hua
identity, then it is alternative.
\end{abstract}

\vspace*{-1.1cm}
\maketitle

\pagestyle{fancy}
\setlength{\headheight}{12pt}
\fancyhead[L]{\footnotesize P. Zusmanovich}
\fancyhead[C]{\footnotesize \MyTitle}
\fancyhead[R]{\footnotesize \thepage{}}
\fancyfoot[L,C,R]{}
%%%%%%%%%%%%%%%%%%%%%%%%%%%%%%%%%%%%%%%%%%%%%%%%%%%%%%%%%%%%%%%%%%%%%%%%%%

\vspace*{-0.7cm}
\section*{Introduction}

In any associative division ring, it holds
\begin{equation*}
a - \big(a^{-1} + (b^{-1} - a)^{-1}\big)^{-1} = aba
\end{equation*}
for any two nonzero elements $a,b$ of the ring such that $ab \ne 1$. This is the
Hua identity which, together with its one-line proof, adorns most of the algebra
textbooks. In fact, a bit more is true: the identity holds in any associative
ring whenever all the involved inverses make sense.

Here we are concerned with the natural question: what about the converse of this
statement? Namely, if a not necessarily associative ring with unit, and such
that any nonzero element has a multiplicative inverse, satisfies the Hua
identity, must it be associative? The immediate answer is ``no'', for an
apparent reason: by an extension of the Artin theorem, any division subring of
an alternative division ring generated by two elements is associative (see,
e.g., \cite[Chapter 7, \S 3, Exercise 3]{nearly-ass}); the Hua identity involves
two elements only, the division subring generated by these two elements is
associative, and so the Hua identity holds for them. Thus, the Hua identity
holds in any alternative division ring. Moreover, like in the associative case,
the Hua identity holds in any alternative ring whenever all the involved
inverses make sense; the few-lines proof is just a bit more involved than in the
associative case and uses flexibility and the right (or left) inverse property
-- see, e.g., \cite[\S 2]{ferreira}. So the converse statement should involve
not associativity, but (a more general) alternativity:

\begin{theorem}\label{th-1}
Let $R$ be a ring with unit, having multiplicative inverses, and satisfying the
Hua identity
\begin{equation}\label{eq-huahua}
a - \big(a^{-1} + (b^{-1} - a)^{-1}\big)^{-1} = (ab)a
\end{equation}
for any nonzero elements $a,b\in R$ such that $a \ne b^{-1}$ and
$a^{-1} + (b^{-1} - a)^{-1} \ne 0$. Then $R$ is alternative. 
\end{theorem}

In other words, for a ring $R$ the following two conditions are
equivalent:
\begin{enumerate}[\upshape(i)]
\item
$R$ is unital, having multiplicative inverses, and satisfies the Hua identity;
\item 
$R$ is an alternative division ring.
\end{enumerate}

Note that when writing the Hua identity in not necessarily associative ring, we
have to make the choice of bracketing at the right-hand side $aba$. This choice
really does not matter: if the statement is true for either choice, say, for
$(ab)a$, as stated in \eqref{eq-huahua}, then applying it for the ring with the
opposite multiplication, we get the same statement with the variant of the Hua
identity having $a(ba)$ at the right-hand side.

Note also that we do not assume uniqueness of inverses; that means that the Hua
identity \eqref{eq-huahua} holds for every admissible choice of the inverses
involved.

We split the proof of Theorem \ref{th-1} into two lemmas.

\begin{lemma}\label{lemma-inv}
Let $R$ be a ring satisfying the conditions of Theorem \ref{th-1}. Then inverses
in $R$ are unique.
\end{lemma}

\begin{lemma}\label{lemma-alt}
Let $R$ be a ring satisfying the conditions of Theorem \ref{th-1}, and,
moreover, inverses in $R$ are unique. Then $R$ is alternative.
\end{lemma}

\section*{Notation and Conventions}

By \emph{ring} we mean an arbitrary ring which is not assumed to be associative,
alternative, or to satisfy any other distinguished identity; the adjectives
``associative'', ``alternative'', etc., are always mentioned explicitly. The
multiplication from the right, respectively from the left, by an element $x$ is
denoted by $r_x$, respectively, $\ell_x$. A (unital) ring is said to have
\emph{multiplicative inverses}, if for any nonzero element $x$ of the ring,
there is -- not necessarily unique -- element $x^{-1}$ such that
$x x^{-1} = x^{-1} x = 1$. Note that in such rings $1^{-1} = 1$ is defined
uniquely. 

A ring $R$ is said to have \emph{right inverse property} if for any 
nonzero $a \in R$ there is $b \in R$ such that $(xa)b = x$ for any $x\in R$ 
(or, in other words, $r_a^{-1} = r_b$); the \emph{left inverse property} is
defined analogously. If a ring has multiplicative inverses, then the right
inverse property can be written in a more precise form: $r_a^{-1} = r_{a^{-1}}$.
A ring $R$ is a \emph{division ring}, if for any $a,b \in R$, $a \ne 0$, there
are unique elements $x,y \in R$ such that $ax = b$ and $ya = b$. 

Compositions of maps are written from right to left.

\section{Proof of Lemma \ref{lemma-inv}, $\modulus{R} > 4$}

For each nonzero $x \in R$, fix arbitrarily a choice of inverse $x^{-1}$, which
will be denoted by $i(x)$. Thus, we have
\begin{equation}\label{eq-hua}
a - i\Big(i(a) + i\big((i(b) - a\big)\Big) = (ab)a
\end{equation}
whenever the expression at the left-hand side makes sense. Moreover, $i(1) = 1$,
and we can assume $i(-x) = -i(x)$ for any nonzero $x\in R$.

Substituting $a=1$ into \eqref{eq-hua}, we get
\begin{equation}\label{eq-i}
i\Big(1 + i\big(i(b) - 1\big)\Big) = 1 - b
\end{equation}
for any $b\in R$ such that $b\ne 0,1$. Note that the left-hand side here is well
defined: $i(b) = 1$ implies $b=1$, and $i\big(i(b) - 1\big) = -1$ implies
$i(b) = 0$, a contradiction.

Now, \eqref{eq-i} shows that $i(b) = i(c)$ implies $b=c$, and the range of
possible values of $i$ contains all elements of $R$ of the form $1-b$ for 
$b\ne 0,1$, that is, $R \backslash \{0,1\}$, as well as $1 =i(1)$. This shows
that the map $i: R \backslash \{0\} \to R \backslash \{0\}$ is both injective
and surjective.

Fix $a\ne 0$. Let us denote the left-hand side of \eqref{eq-hua} by $H_a(b)$,
and consider it as a function of $b$. It is defined on the whole $R$ except for
the values $b=0$, $b=i^{-1}(a)$, and the value of $b$ satisfying
$i(i(b) - a) = -i(a) = i(-a)$. Due to bijectivity of $i$, the latter equality
yields $i(b)=0$, a contradiction.

Thus, the map $H_a$ is defined on $R \backslash \{0,i^{-1}(a)\}$. It is a
composition of the maps $i$, the change-of-the-sign map $x \mapsto -x$, and 
translations of the form $x \mapsto x + c$ for certain fixed element $c$ (equal
to either $a$, or to $i(a)$). All these maps are injective, hence $H_a$ is
injective. But \eqref{eq-hua} yields $H_a(b) = (ab)a$, so we have (an additive)
map $U_a: b \mapsto (ab)a$ (obviously, defined on the whole $R$)\footnote{
The maps $U_a$, sometimes called ``quadratic multiplications'', play a certain
role in the theory of alternative algebras and their relation to Jordan
algebras; see, e.g., \cite{mccrimmon}.
},
which is injective on $R \backslash \{0,i^{-1}(a)\}$.

Assume $R$ contains more than four elements. Let $b$ be a nonzero element in
$\Ker(U_a)$. Choose $c\in R$ different from any of the four elements $0$,
$i^{-1}(a)$, $-b$, and $i^{-1}(a) - b$. Then
$c, b+c \in R \backslash \{0,i^{-1}(a)\}$, and $U_a(c) = U_a(b+c)$, a
contradiction. Thus $\Ker(U_a) = 0$, and $U_a$ is injective
on the whole $R$.

We have $U_a = r_a \circ \ell_a$. Since $U_a$ is injective, $\ell_a$ is
injective too. In particular, $R$ has no zero divisors, hence the multiplicative
inverses are unique, i.e., $x^{-1} = i(x)$ for any nonzero $x\in R$.

\section{Proof of Lemma \ref{lemma-inv}, $\modulus{R} \le 4$}

It remains to consider the case where $R$ contains at most four elements. The
classification of all (nonassociative) rings of order $\le 4$ is available, see
\cite{boers}; so, in principle, this case can be resolved by a direct inspection
of the list there containing 52 entries. However, this is an arduous task, even
at the age of AI, so we can do much better by employing the standard arguments
from elementary algebra.

We are going to prove that a unital ring of order $\le 4$, and having 
multiplicative inverses, is associative. This will establish
Lemma~\ref{lemma-inv} (and, in fact, the whole Theorem \ref{th-1}) in this case.

The characteristic of a ring with multiplicative inverses is zero or prime; the
proof is the same as in the case of fields (uniqueness of inverses is not
needed). Since the ring is finite, the characteristic is prime. Then the subring of $R$ generated by $1$, let us denote it by $F$, is the prime
subfield, and $R$ has an algebra structure over $F$. If the order of $R$ is
prime, it coincides with $F$. Thus the only remaining case to consider is
$\modulus{R} = 4$. In this case the characteristic of $R$ is $2$, $F = GF(2)$,
and $R$ is a two-dimensional algebra over $GF(2)$. But any two-dimensional unital
algebra $R$ over a field $F$ is associative (even without the assumptions of $R$
having multiplicative inverses, and $F$ coinciding with $GF(2)$). Indeed, choose a basis $\{1, a\}$ in $R$. Then $a^2 = \alpha 1 + \beta a$ for certain
$\alpha, \beta \in F$, and thus $R$ is isomorphic to
$F[t]/(t^2 - \beta t -\alpha)$, an associative (and commutative) algebra.

\section{Proof of Lemma \ref{lemma-alt}}

We pick up the narrative from the proof of Lemma \ref{lemma-inv}. 

Extend $R$ to the set $R \cup \{\infty\}$, and extend the (unique) inverse map,
the change-of-the-sign map, and translations by setting
$0^{-1} = \infty$, $\infty^{-1} = 0$, $-\infty = \infty$, and 
$\infty + c = c + \infty = \infty$ for any $c\in R$. Then $H_a(0) = 0$,
$H_a(a^{-1}) = a$, and $H_a(\infty) = \infty$, so $H_a$ is a bijection on
$R \cup \{\infty\}$, which, being restricted to $R$, coincides with $U_a$.
Consequently, $U_a$ is a bijection, and both $\ell_a$ and $r_a$ are bijections.
This means that $R$ is a division ring.

Now take $a,b \in R$ such that $a,b,a+b \ne 0$. Writing the Hua identity for the
pairs $(a+b,a^{-1})$ and $(b,-a^{-1})$ gives respectively:
$$
(a+b) - \big((a+b)^{-1} - b^{-1}\big)^{-1} = a + b + (ba^{-1})a + (ba^{-1})b
$$
and
$$
b + \big((a+b)^{-1} - b^{-1}\big)^{-1} = -(ba^{-1})b .
$$

Summing up the last two equalities yields
$$
(ba^{-1})a = b .
$$

The last equality, in its turn, obviously holds also in the excluded cases $b=0$
and $b=-a$, thus it holds for any $a,b \in R$, $a\ne 0$. Interchanging $a$ and
$a^{-1}$, we get
$$
(ba)a^{-1} = b ,
$$
i.e., $R$ has the right inverse property. But then $R$ is right alternative 
(right Moufang in characteristic $2$, see, e.g., \cite[\S 1]{mccrimmon}). But by
the Skornyakov--San Soucie theorem, a right alternative (also satisfying the
right Moufang identity in characteristic $2$) division ring is alternative (see,
e.g., \cite[Chapter 16, \S 1, Corollary 2]{nearly-ass}, and \cite{san-soucie}
for characteristic $2$ case).

\section{Epilogue: The Hua identity strikes again}

One interesting aspect of the just presented proof is left implicit: this is the
penultimate implication, claiming right alternativity of a division ring with
the right inverse property. Somewhat amusingly, the arguments involves the
application of the associative Hua identity, and we reproduce it here. The 
arguments are certainly not new, see, e.g., the already cited
\cite[\S 1]{mccrimmon}, or \cite[Lemma 3.1]{darpo}.

First observe:

\begin{theorem}\label{th-2}
Let $R$ be a unital ring having multiplicative inverses. Then the following two
conditions are equivalent:
\begin{enumerate}[\upshape(i)]
\item $R$ has the right inverse property;
\item $R$ satisfies the Hua identity \eqref{eq-huahua} .
\end{enumerate}
\end{theorem}

\begin{proof}
(i) $\Rightarrow$ (ii).
Fix nonzero elements $a,b$ of a ring $R$ such that $a \ne b^{-1}$ and
$a^{-1} + (b^{-1} - a)^{-1} \ne 0$, and apply the Hua identity to the elements $r_a$, $r_b$ of the (associative) multiplication ring of $R$:
\begin{equation}\label{eq-rr}
r_a - \big(r_a^{-1} + (r_b^{-1} - r_a)^{-1}\big)^{-1} = 
r_a \circ r_b \circ r_a .
\end{equation}

Using the fact that $r_a^{-1} = r_{a^{-1}}$ and the additivity of $r$, the last
equality can be rewritten as
\begin{equation}\label{eq-r}
r_{a - (a^{-1} + (b^{-1} - a)^{-1})^{-1}} = r_a \circ r_b \circ r_a .
\end{equation}

(This also shows that all the terms at the left-hand side of \eqref{eq-rr} are
well defined). Evaluating both sides of the last equality at $1$, we get the Hua
identity
\eqref{eq-huahua} for $R$.

(ii) $\Rightarrow$ (i). As in the proof of Theorem \ref{th-1}.
\end{proof}

Now, assume a ring $R$ satisfies the equivalent conditions of
Theorem~\ref{th-2}. The Hua identity for $R$ coupled with \eqref{eq-r} gives
\begin{equation}\label{eq-aba}
r_{(ab)a} = r_a \circ r_b \circ r_a
\end{equation}
for any nonzero $a,b \in R$ such that $a \ne b^{-1}$ and
$a^{-1} + (b^{-1} - a)^{-1} \ne 0$. By Lemma \ref{lemma-inv}, the inverses in
$R$ are unique, hence the latter inequality holds always. The equality
\eqref{eq-aba} holds trivially also when one of $a,b$ is zero, and in the
remaining case $a = b^{-1}$ it holds due to the right inverse property. Thus \eqref{eq-aba}
is an identity in $R$; this is the right Moufang identity
$$
c((ab)a) = (((ca)b)a) ,
$$
which in characteristic $\ne 2$ is equivalent to right alternativity\footnote{
The right Moufang identity also sometimes called the \emph{right Bol identity}
in the literature. Rings satisfying both the right alternative identity and the right
Moufang (Bol) identity are sometimes called \emph{strongly right alternative}.
In characteristic $\ne 2$, the classes of alternative rings and strongly right
alternative rings coincide; see, e.g., the already cited
\cite[Chapter 16, \S 1]{nearly-ass} and \cite{san-soucie}.
}.

Thus, the full proof of Theorem~\ref{th-1} uses the Hua identity twice: first
the ``nonassociative'' version \eqref{eq-huahua} as given in the theorem
premises, and, second, the (classical) associative version applied to the
multiplication ring of the ring in question!

\section*{Acknowledgement}

ChatGPT (GPT 5.6-Sol and 6-Astra) was used.


\begin{thebibliography}{ZSSS}

\bibitem[B]{boers} A.H. Boers, \emph{L'anneau \'a quatre elements},
Proc. Koninkl. Ned. Akad. Wet. Ser. A \textbf{69} (1966), 14--21.

%\bibitem[DPI1]{darpo-arxiv} E. Darp\"o, J.M. P\'erez Izquierdo,
%\emph{Inversion and quasigroup identities in division algebras},
%arXiv:1205.6250.

\bibitem[DPI]{darpo} E. Darp\"o, J.M. P\'erez Izquierdo,
\emph{Autotopies and quasigroup identities: New aspects of non-associative division algebras},
Forum Math. \textbf{27} (2015), no.5, 2691--2745.

\bibitem[FJS]{ferreira} B.L.M. Ferreira, H. Julius, D. Smigly,
\emph{Commuting maps and identities with inverses on alternative division rings},
J. Algebra \textbf{638} (2024), 488--505.

\bibitem[M]{mccrimmon} K. McCrimmon,
\emph{Quadratic methods in nonassociative algebras},
Proc. ICM, Vancouver 1974, Vol. 1, 325--330.

\bibitem[SS]{san-soucie} R.L. San Soucie,
\emph{Right alternative division rings of characteristic two},
Proc. Amer. Math. Soc. \textbf{6} (1955), no.2, 291--296.

\bibitem[ZSSS]{nearly-ass} 
K.A. Zhevlakov, A.M. Slin'ko, I.P. Shestakov, A.I. Shirshov,
\emph{Rings That Are Nearly Associative}, Academic Press, 1982.

\end{thebibliography}
\end{document}